\documentclass[12pt]{amsart}
\usepackage{amsmath,amsthm,amsfonts,amssymb,mathrsfs}
\date{\today}
\usepackage{color}

\newtheorem{theorem}{Theorem}[section]
\newtheorem{proposition}[theorem]{Proposition}
\newtheorem{corollary}[theorem]{Corollary}
\newtheorem{lemma}[theorem]{Lemma}

\theoremstyle{definition}
\newtheorem{example}[theorem]{Example}
\newtheorem{remark}[theorem]{Remark}
\newtheorem{definition}[theorem]{Definition}

\usepackage{hyperref}

\begin{document}

\title[On some generalization of the bicyclic semigroup: the topological version]{On some generalization of the bicyclic semigroup: the topological version}

\author[M.~Cencelj]{Matija Cencelj}
\address{Faculty of Education, University of Ljubljana, Kardeljeva Pl. 16,
Ljubljana, 1000, Slovenia} \email{matija.cencelj@pef.uni-lj.si}

\author[O.~Gutik]{Oleg~Gutik}
\address{Department of Mechanics and Mathematics, Ivan Franko Lviv National
University, Universytetska 1, Lviv, 79000, Ukraine}
\email{{o\_gutik}@franko.lviv.ua, {ovgutik}@yahoo.com}

\author[D.~Repov\v{s}]{Du\v{s}an~Repov\v{s}}
\address{Faculty of Education, and
Faculty of Mathematics and Physics, University of Ljubljana, Kardeljeva Pl. 16,
Ljubljana, 1000, Slovenia} \email{dusan.repovs@guest.arnes.si}


\begin{abstract}
We show that every Hausdorff Baire topology $\tau$ on $\mathcal{C}=\langle a,b\mid a^2b=a, ab^2=b\rangle$ such that $(\mathcal{C},\tau)$ is a semitopological semigroup is discrete and we construct a nondiscrete Hausdorff semigroup topology on $\mathcal{C}$. We also discuss the closure of a semigroup $\mathcal{C}$ in a semitopological semigroup and prove that $\mathcal{C}$ does not embed into a topological semigroup with the countably compact square.
\end{abstract}

\maketitle

\section{Introduction and preliminaries}\label{s1}

In this paper all topological spaces are assumed to be Hausdorff.  If $Y$ is a subspace of a topological space $X$ and $A\subseteq Y$, then we shall denote the topological closure of $A$ in $Y$ by $\operatorname{cl}_Y(A)$. Further we shall follow the terminology of \cite{CHK, Clifford-Preston-1961-1967,
Engelking1989, Ruppert1984}.

For a topological space $X$, a family $\{A_s\mid s\in\mathscr{A}\}$ of subsets of $X$ is called \emph{locally finite} if for every point $x\in X$ there exists an open neighbourhood $U$ of $x$ in $X$ such that the set $\{s\in\mathscr{A}\mid U\cap A_s\neq\varnothing\}$ is finite. A subset $A$ of $X$ is said to be
\begin{itemize}
  \item \emph{co-dense} in $X$ if $X\setminus A$ is dense in $X$;
  \item an \emph{$F_\sigma$-set} in $X$ if $A$ is a union of a countable family of closed subsets in $X$.
\end{itemize}

We recall that a topological space $X$ is said to be
\begin{itemize}
  \item \emph{compact} if each open cover of $X$ has a finite subcover;
  \item \emph{countably compact} if each open countable cover of $X$ has a finite subcover;
  \item \emph{sequentially compact} if each sequence in $X$ has a convergent subsequence;
  \item \emph{pseudocompact} if each locally finite open cover of $X$ is finite;
  \item a \emph{Baire space} if for each sequence $A_1, A_2,\ldots, A_i,\ldots$ of nowhere dense subsets of $X$ the union $\displaystyle\bigcup_{i=1}^\infty A_i$ is a co-dense subset of $X$;
  \item  \emph{\v{C}ech complete} if $X$ is Tychonoff and for every compactification $cX$ of $X$, the remainder $cX\setminus X$ is an $F_\sigma$-set in $cX$;
  \item  \emph{locally compact} if every point of $X$ has an open neighbourhood with a compact closure.
\end{itemize}
According to Theorem~3.10.22 of \cite{Engelking1989}, a Tychonoff
topological space $X$ is pseudocompact if and only if each
continuous real-valued function on $X$ is bounded.

If $S$ is a semigroup, then we shall denote the {\it Green relations}
on
$S$ by $\mathscr{R}$ and $\mathscr{L}$  (see Section~2.1 of \cite{Clifford-Preston-1961-1967}):
\begin{align*}
    &\qquad a\mathscr{R}b \mbox{ if and only if } aS^1=bS^1; \qquad \hbox{and} \qquad \qquad a\mathscr{L}b \mbox{ if and only if } S^1a=S^1b.
\end{align*}

A semigroup $S$ is called \emph{simple} if $S$ does not contain any proper two-sided ideals.

A {\it semitopological} (resp. \emph{topological}) {\it semigroup}
is a topological space together with a separately (resp. jointly)
continuous semigroup operation.

An important theorem of Andersen \cite{Andersen} (see also \cite[Theorem~2.54]{Clifford-Preston-1961-1967}) states that in any [$0$-]simple semigroup which is not completely [$0$-]simple, each nonzero idempotent (if there are any) is the identity element of a copy of the bicyclic semigroup $\mathcal{B}(a,b)= \langle a,b \mid ab=1\rangle$. The bicyclic semigroup is bisimple and every one of its congruences is either trivial or a group congruence. Moreover, every non-annihilating homomorphism $h$ of the bicyclic
semigroup is either an isomorphism or the image of $\mathcal{B}(a,b)$ under $h$ is a cyclic group (see Corollary~1.32 in \cite{Clifford-Preston-1961-1967}). Eberhart and Selden  \cite{EberhartSelden1969} showed that every Hausdorff semigroup topology on the bicyclic semigroup $\mathcal{B}(a,b)$ is discrete. Bertman and West \cite{BertmanWest1976} proved that every Hausdorff topology $\tau$ on $\mathcal{B}(a,b)$ such that $(\mathcal{B}(a,b),\tau)$ is a semitopological semigroup is also discrete.  Neither stable nor $\Gamma$-compact
topological semigroups can contain a copy of the bicyclic semigroup~\cite{AHK, HildebrantKoch1986}. Also, the bicyclic semigroup cannot be embedded into any countably compact topological inverse semigroup~\cite{GutikRepovs2007}. Moreover, the conditions were given in \cite{BanakhDimitrovaGutik2009} and \cite{BanakhDimitrovaGutik2010} when a countably compact or pseudocompact topological semigroup cannot contain the bicyclic semigroup, which is topological semigroup with a countably compact square and with a pseudocompact square. However, Banakh, Dimitrova and Gutik~\cite{BanakhDimitrovaGutik2010} have constructed (assuming the Continuum Hypothesis or the Martin Axiom) an example of a Tychonoff countably compact topological semigroup which contains the bicyclic semigroup.

Jones \cite{Jones1987} found semigroups $\mathcal{A}$ and $\mathcal{C}$ which play a role similar to the bicyclic semigroup in Andersen's Theorem. Let  $\mathcal{A}=\langle a,b\mid a^2b=a\rangle$ and $\mathcal{C}=\langle a,b\mid a^2b=a, ab^2=b\rangle$. It is obvious that the semigroup $\mathcal{C}$ is a homomorphic image of $\mathcal{A}$, and the bicyclic semigroup is a homomorphic image of $\mathcal{C}$. Also, every non-injective homomorphic image of the semigroup $\mathcal{C}$ contains an idempotent. Jones \cite{Jones1987} showed that every [$0$-] simple idempotent-free semigroup $S$ on which $\mathscr{R}$ is nontrivial contains (a copy of) $\mathcal{A}$ or $\mathcal{C}$. Moreover, if $S$ is also $\mathscr{L}$-trivial and is not $\mathscr{R}$-trivial then it must contain $\mathcal{A}$ (but not $\mathcal{C}$), and if $S$ is both $\mathscr{R}$- and $\mathscr{L}$-nontrivial then $S$ must contain either $\mathcal{C}$ or both $\mathcal{A}$ and its dual $\mathcal{A}^d$.

In the general case, the countable compactness of topological semigroup $S$ does not guarantee that $S$ contains an idempotent. By Theorem~8 of~\cite{BanakhDimitrovaGutik2009}, a topological semigroup $S$ contains an idempotent if $S$ satisfies one of the following conditions: 1)~$S$ is doubly countably compact; 2)~$S$ is sequentially compact; 3)~$S$ is $p$-compact for some free ultrafilter $p$ on $\omega$; 4)~$S^{2^{\mathfrak c}}$ is countably compact; 5)~$S^{\kappa^\omega}$ is countably compact, where $\kappa$ is the minimal cardinality of a closed subsemigroup of $S$. This motivates the establishing of the semigroups $\mathcal{A}$ and $\mathcal{C}$ as topological semigroups, in particular their semigroup topologizations and the question of their embeddings into compact-like topological semigroups.

In this paper we study the semigroup $\mathcal{C}$ as a semitopological semigroup. We show that every Hausdorff Baire topology $\tau$ on $\mathcal{C}$ such that $(\mathcal{C},\tau)$ is a semitopological semigroup is discrete and we construct a nondiscrete Hausdorff semigroup topology on $\mathcal{C}$. We also discuss the closure of a semigroup $\mathcal{C}$ in a semitopological semigroup and prove that $\mathcal{C}$ does not embed into a topological semigroup with a countably compact square.


\section{Algebraic properties of the semigroup $C$}\label{s2}

The semigroup $\mathcal{C}=\langle a,b\mid a^2b=a, ab^2=b\rangle$ was introduced by R\'{e}dei \cite{Redei1960} and further studied by Megyesi and Poll\'{a}k \cite{Megyesi-Pollak1981} and  by Rankin and Reis \cite{Rankin-Reis1980}. Its salient properties are summarized here:

\begin{proposition}\label{proposition-2.1}
\begin{itemize}
  \item[$(i)$] $\mathcal{C}$ is a $2$-generated simple idempotent-free semigroup in which
      $a\mathscr{R}a^2$ and $b\mathscr{L}b^2$, so that $\mathscr{R}$ and $\mathscr{L}$ are nontrivial; however $\mathscr{H}$ is trivial.
  \item[$(ii)$] Each element of $\mathcal{C}$ is uniquely expressible as $b^k(ab)^la^m$, $k,l,m\geqslant 0$, $k+l+m>0$.
  \item[$(iii)$] The product of elements $b^k(ab)^la^m$ and $b^n(ab)^pa^q$ in $\mathcal{C}$ is equal to
\begin{equation}\label{eq-2.1}
\left\{
  \begin{array}{ll}
    b^{k+n-m}(ab)^pa^q, & \hbox{if~} m<n;\\
    b^{k}(ab)^{l+p+1}a^q, & \hbox{if~} m=n\neq 0;\\
    b^{k}(ab)^{l+p}a^q, & \hbox{if~} m=n=0;\\
    b^k(ab)^la^{q+m-n}, & \hbox{if~} m>n.
  \end{array}
\right.
\end{equation}
  \item[$(iv)$] The semigroup $\mathcal{C}$ is minimally idempotent-free (i.e., it is idempotent-free but each of its proper quotients contains an idempotent).
\end{itemize}
\end{proposition}

\begin{definition}[{\cite{Koch-Wallace1957}}]\label{definition-2.2}
A semigroup $S$ is said to be \emph{stable} if the following conditions hold:
\begin{itemize}
  \item[$(i)$] $s,t\in S$ and $Ss\subseteq Sst$ implies that $Ss=Sst$; \; and
  \item[$(ii)$] $s,t\in S$ and $sS\subseteq tsS$ implies that $sS=tsS$.
\end{itemize}
\end{definition}

By formula (\ref{eq-2.1}) we have that
\begin{equation*}
    b\cdot b^n(ab)^pa^q=b^{n+1}(ab)^pa^q \qquad \hbox{and} \qquad a\cdot b\cdot b^n(ab)^pa^q=
    \left\{
      \begin{array}{ll}
        (ab)^{p+1}a^q, & \hbox{if~} n=0;\\
        b^n(ab)^pa^q, & \hbox{if~} n\geqslant 1,
      \end{array}
    \right.
\end{equation*}
for each $b^n(ab)^pa^q\in\mathcal{C}$. Hence we get that $b\cdot\mathcal{C}\subseteq a\cdot b\cdot\mathcal{C}$, but $b\cdot\mathcal{C}\neq a\cdot b\cdot\mathcal{C}$. This yields the following proposition:

\begin{proposition}\label{proposition-2.3}
The semigroup $\mathcal{C}$ is not stable.
\end{proposition}

The following remark follows from formula (\ref{eq-2.1}) above:

\begin{remark}\label{remark-2.10a}
The semigroup operation in $\mathcal{C}$ implies that the following assertions hold:
\begin{itemize}
  \item[$(i)$] The map $\varphi_{i,j}\colon\mathcal{C}\rightarrow\mathcal{C}$ defined by the formula $\varphi_{i,j}(x)=b^i\cdot x\cdot a^j$ is injective for all nonnegative integers $i$ and $j$ (for $i=j=0$ we put that $\varphi_{0,0}(x)=x$);
  \item[$(ii)$] The subsemigroups $\mathcal{C}_{ab}=\langle ab\rangle$, $\mathcal{C}_{a}=\langle a\rangle$ and $\mathcal{C}_{b}=\langle b\rangle$ in $\mathcal{C}$ are infinite cyclic semigroups.
\end{itemize}
\end{remark}


\section{On topologizations of the semigroup $\mathcal{C}$}\label{s3}

Let $X$ be a topological space. A continuous map $f\colon X\rightarrow X$ is called a \emph{retraction} of $X$ if $f\circ f=f$; and the set of all values of a retraction of $X$ is called a \emph{retract} of $X$ (cf. \cite{Engelking1989}).

\begin{proposition}\label{proposition-3.1}
If $\tau$ is a Hausdorff topology on $\mathcal{C}$ such that $(\mathcal{C},\tau)$ is a semitopological semigroup then for every positive integer $k$ the sets
\begin{align*}
    &\mathfrak{R}_k=\left\{b^n(ab)^pa^q\mid n=k,k+1,k+2,\ldots, p=0,1,2,\ldots, q=0,1,2,\ldots\right\}, \; \hbox{and}\\
    &\mathfrak{L}_k=\left\{b^n(ab)^pa^q\mid q=k,k+1,k+2,\ldots, n=0,1,2,\ldots, p=0,1,2,\ldots\right\}
\end{align*}
are retracts in $(\mathcal{C},\tau)$ and hence closed subsets of $(\mathcal{C},\tau)$.
\end{proposition}

\begin{proof}
By formula (\ref{eq-2.1}) we have that
\begin{gather}
    b^m(ab)^la^m\cdot b^n(ab)^pa^q=
\left\{
  \begin{array}{ll}
    b^{n}(ab)^pa^q,       & \hbox{if~} m<n;\\
    b^{n}(ab)^{l+p+1}a^q, & \hbox{if~} m=n\neq 0;\\
    (ab)^{l+p}a^q,        & \hbox{if~} m=n=0;\\
    b^m(ab)^la^{q+m-n},   & \hbox{if~} m>n,\\
  \end{array}
\right.\label{eq-3.1a}
    \\
    b^i(ab)^la^m\cdot b^n(ab)^pa^n=
\left\{
  \begin{array}{ll}
    b^{i+n-m}(ab)^pa^n,   & \hbox{if~} m<n;\\
    b^{i}(ab)^{l+p+1}a^n, & \hbox{if~} m=n\neq 0;\\
    b^{i}(ab)^{l+p},      & \hbox{if~} m=n=0;\\
    b^i(ab)^la^{m},       & \hbox{if~} m>n.
  \end{array}
\right.\label{eq-3.1b}
\end{gather}
Then left and right translations of the element $b^k(ab)^la^k$ of the semigroup $\mathcal{C}$ are retractions of the topological space $(\mathcal{C},\tau)$ and hence the sets $\mathfrak{R}_k$ and $\mathfrak{L}_k$ are retracts of the topological space $(\mathcal{C},\tau)$ for every positive integer $k$. The last statement of the proposition follows from Exercise~1.5.C of \cite{Engelking1989}.
\end{proof}

\begin{proposition}\label{proposition-3.2}
If $\tau$ is a Hausdorff topology on $\mathcal{C}$ such that $(\mathcal{C},\tau)$ is a semitopological semigroup then $\mathcal{C}_{ab}$ is an open-and-closed subsemigroup of $(\mathcal{C},\tau)$.
\end{proposition}

\begin{proof}
We observe that $\mathcal{C}_{ab}=\mathcal{C}\setminus\left(\mathfrak{R}_1\cup \mathfrak{L}_1\right)$ and hence by Proposition~\ref{proposition-3.1} we have that $\mathcal{C}_{ab}$ is an open subset of $(\mathcal{C},\tau)$. Also, formula (\ref{eq-2.1}) implies that
\begin{equation}\label{eq-3.1}
    a\cdot b^n(ab)^pa^q\cdot b=
\left\{
  \begin{array}{ll}
    b^{n-1}(ab)^pa^q\cdot b, & \hbox{if~} n>1;\\
    (ab)^{p+l}a^q\cdot b,    & \hbox{if~} n=1;\\
    a^{q+1}\cdot b,          & \hbox{if~} n=0
  \end{array}
\right.
=
\left\{
  \begin{array}{lll}
    b^{n},                & \hbox{if~} n>1 & \hbox{and~} q=0;\\
    b^{n-1}(ab)^{p+1}     & \hbox{if~} n>1 & \hbox{and~} q=1;\\
    b^{n-1}(ab)^pa^{q-1}, & \hbox{if~} n>1 & \hbox{and~} q>1;\\
    b,                    & \hbox{if~} n=1 & \hbox{and~} q=0;\\
    (ab)^{p+2},           & \hbox{if~} n=1 & \hbox{and~} q=1;\\
    (ab)^{p+1}a^{q-1},    & \hbox{if~} n=1 & \hbox{and~} q>1;\\
    ab,                   & \hbox{if~} n=0 & \hbox{and~} q=0;\\
    a,                    & \hbox{if~} n=0 & \hbox{and~} q=1;\\
    a^q,                  & \hbox{if~} n=0 & \hbox{and~} q>1,
  \end{array}
\right.
\end{equation}
for nonnegative integers $n,p$ and $q$. By formula (\ref{eq-3.1}), $\mathcal{C}_{0,0}=\left\{(ab)^i\mid i=1,2,3,\ldots\right\}$ is the set of solutions of the equation $a\cdot X\cdot b=ab$. Then the Hausdorffness of the space $(\mathcal{C},\tau)$ and the separate continuity of the semigroup operation in $\mathcal{C}$ imply that $\mathcal{C}_{ab}=\mathcal{C}_{0,0}$ is a closed subset of $(\mathcal{C},\tau)$.
\end{proof}

We observe that  formula (\ref{eq-3.1}) implies that
\begin{gather}
    b^k(ab)^la^m\cdot b=
\left\{
  \begin{array}{ll}
    b^{k+1},          & \hbox{if~} m=0;\\
    b^{k}(ab)^{l+1},  & \hbox{if~} m=1;\\
    b^k(ab)^la^{m-1}, & \hbox{if~} m>1,\\
  \end{array}
\right.\label{eq-3.2}\\
a\cdot b^n(ab)^pa^q=
\left\{
  \begin{array}{ll}
    b^{n-1}(ab)^pa^q, & \hbox{if~} n>1;\\
    (ab)^{p+1}a^q,    & \hbox{if~} n=1;\\
    a^{q+1},          & \hbox{if~} n=0,
  \end{array}
\right.\label{eq-3.3}
\end{gather}
for nonnegative integers $k,l,m,n,p$ and $q$.

\begin{proposition}\label{proposition-3.3}
If $\tau$ is a Hausdorff topology on $\mathcal{C}$ such that $(\mathcal{C},\tau)$ is a semitopological semigroup then $\mathcal{C}_{0,i}=\left\{(ab)^pa^i\mid p=0,1,2,3,\ldots\right\}$ and $\mathcal{C}_{i,0}=\left\{b^i(ab)^p\mid p=0,1,2,3,\ldots\right\}$ are open subsets of $(\mathcal{C},\tau)$ for any positive integer $i$.
\end{proposition}

\begin{proof}
By Proposition~\ref{proposition-3.2}, $\mathcal{C}_{0,0}$ is an open subset $(\mathcal{C},\tau)$ and by Hausdorffness of $(\mathcal{C},\tau)$ the set $\mathcal{C}_{0,0}\setminus\left\{ab\right\}$ is open in $(\mathcal{C},\tau)$, too. Then formula (\ref{eq-3.2}) implies that the equation $X\cdot b=(ab)^{p+2}$, where $p=0,1,2,3,\ldots$, has a unique solution $X=(ab)^pa$, and hence since all right translations in $(\mathcal{C},\tau)$ are continuous maps we get that $\mathcal{C}_{0,1}$ is an open subset of the topological space $(\mathcal{C},\tau)$. Also, formula (\ref{eq-3.1}) implies that the equation $a\cdot X=(ab)^{p+2}$, where $p=0,1,2,3,\ldots$, has a unique solution $X=b(ab)^p$, and hence since all left translations in $(\mathcal{C},\tau)$ are continuous maps we get that $\mathcal{C}_{1,0}$ is an open subset of the topological space $(\mathcal{C},\tau)$.

By formula (\ref{eq-3.2}), the equation $X\cdot b=(ab)^la^{m-1}$, where $l-1$ and $m-1$ are positive integers, has a unique solution $X=(ab)^la^{m}$. Then the separate continuity of the semigroup operation in $(\mathcal{C},\tau)$ implies that if $\mathcal{C}_{0,m-1}$ is an open subset of $(\mathcal{C},\tau)$ then $\mathcal{C}_{0,m}$ is open in $(\mathcal{C},\tau)$, too. Similarly, formula (\ref{eq-3.3}) implies that the equation $a\cdot X=b^{n-1}(ab)^p$, where $n-1$ and $p-1$ are positive integers, has a unique solution $X=b^{n}(ab)^p$, and hence the separate continuity of the semigroup operation in $(\mathcal{C},\tau)$ and openess of the set $\mathcal{C}_{n-1,0}$ in $(\mathcal{C},\tau)$ imply that the set $\mathcal{C}_{n,0}$ is an open subset of the topological space $(\mathcal{C},\tau)$. Next, we complete the proof of the proposition by induction.
\end{proof}

\begin{proposition}\label{proposition-3.4}
If $\tau$ is a Hausdorff topology on $\mathcal{C}$ such that $(\mathcal{C},\tau)$ is a semitopological semigroup then $\mathcal{C}_{i,j}=\left\{b^i(ab)^pa^j\mid p=0,1,2,3,\ldots\right\}$ is an open subset of $(\mathcal{C},\tau)$ for all positive integers $i$ and $j$.
\end{proposition}

\begin{proof}
First we observe that Proposition~\ref{proposition-3.3} and Hausdorffness of $(\mathcal{C},\tau)$ imply that $\mathcal{C}_{k,0}\setminus\left\{b^k(ab)\right\}$ is an open subset of $(\mathcal{C},\tau)$ for every positive integer $k$. Then formula (\ref{eq-3.2}) implies that the equation $X\cdot b=b^k(ab)^{p+1}$, where $p=0,1,2,3,\ldots$, has a unique solution $X=b^k(ab)^pa$, and hence since all right and left translations in $(\mathcal{C},\tau)$ are continuous maps we get that $\mathcal{C}_{k,1}$ is an open subset of the topological space $(\mathcal{C},\tau)$.

Also, by formula (\ref{eq-3.2}) we have that the equation $X\cdot b=b^k(ab)^{p}a^l$ has a unique solution $X=b^k(ab)^pa^{l+1}$. Then the openess of the set $\mathcal{C}_{k,l}$ implies that the set $\mathcal{C}_{k,l+1}$ is open in $(\mathcal{C},\tau)$. Then induction implies the assertion of the proposition.
\end{proof}

Propositions~\ref{proposition-3.2}, \ref{proposition-3.3} and \ref{proposition-3.4} imply Theorem~\ref{theorem-3.5}, which describes all Hausdorff topologies $\tau$ on $\mathcal{C}$ such that $(\mathcal{C},\tau)$ is a semitopological semigroup.

\begin{theorem}\label{theorem-3.5}
If $\tau$ is a Hausdorff topology on $\mathcal{C}$ such that $(\mathcal{C},\tau)$ is a semitopological semigroup then $\mathcal{C}_{i,j}$ is an open-and-closed subset of $(\mathcal{C},\tau)$ for all nonnegative integers $i$ and $j$.
\end{theorem}

Since the bicyclic semigroup $\mathcal{B}(a,b)$ admits only the discrete topology which turns $\mathcal{B}(a,b)$ into a Hausdorff semitopological semigroup \cite{BertmanWest1976}, Theorem~\ref{theorem-3.5} implies the following:

\begin{corollary}\label{corollary-3.6}
If $\mathcal{C}$ is a semitopological semigroup then the homomorphism $h\colon\mathcal{C}\rightarrow\mathcal{B}(a,b)$, defined by the formula $h\left(b^k(ab)^la^m\right)=b^ka^m$, is continuous.
\end{corollary}

Later we shall need the following lemma.

\begin{lemma}\label{lemma-3.7}
Every Hausdorff Baire topology on the infinite cyclic semigroup $S$ such that $(S,\tau)$ is a semitopological semigroup is discrete.
\end{lemma}

\begin{proof}
Since every infinite cyclic semigroup is isomorphic to the additive semigroup of positive integers $(\mathbb{N},+)$ we assume without loss of generality that $S=(\mathbb{N},+)$.

Fix an arbitrary $n_0\in\mathbb{N}$. Then Hausdorffness of $(\mathbb{N},+)$ implies that $\{1,\ldots,n_0\}$ is a closed subset of $(\mathbb{N},+)$, and hence by Proposition~1.14 of \cite{Haworth-McCoy1977} we get that $\mathbb{N}_{n_0}=\mathbb{N}\setminus\{1,\ldots,n_0\}$ with the induced topology from $(\mathbb{N},\tau)$ is a Baire space.

If no point in $\mathbb{N}_{n_0}$ is isolated,  then since $(\mathbb{N},\tau)$ is Hausdorff, it follows that $\{n\}$ is nowhere dense in $\mathbb{N}_{n_0}$ for all $n>n_0$. But, if this is the case, then since the space $(\mathbb{N},\tau)$ is countable we conclude that $\mathbb{N}_{n_0}$ cannot be a Baire space. Hence $\mathbb{N}_{n_0}$ contains an isolated point $n_1$ in $\mathbb{N}_{n_0}$. Then the separate continuity of the semigroup operation in $(\mathbb{N},+,\tau)$ implies that $n_0$ is an isolated point in $(\mathbb{N},\tau)$, because $n_1=n_0+(\underbrace{1+\ldots+1}_{(n_1{-}n_0)\hbox{\small{-times}}})$. This completes the proof of the lemma.
\end{proof}

\begin{theorem}\label{theorem-3.8}
Every Hausdorff Baire topology $\tau$ on $\mathcal{C}$ such that $(\mathcal{C},\tau)$ is a semitopological semigroup is discrete.
\end{theorem}

\begin{proof}
By Proposition~\ref{proposition-3.2}, $\mathcal{C}_{ab}$ is an open-and-closed subsemigroup of $(\mathcal{C},\tau)$. Then by Proposition~1.14 of \cite{Haworth-McCoy1977} we have that $\mathcal{C}_{ab}$ is a Baire space and hence Lemma~\ref{lemma-3.7} implies that every element of $\mathcal{C}_{ab}$ is an isolated point of the topological space $(\mathcal{C},\tau)$.

Now, by formula (\ref{eq-3.1}), the equation $a\cdot X\cdot b=(ab)^{p+2}$ has a unique solution $X=b(ab)^pa$ for every nonnegative integer $p$, and since the semigroup operation in $(\mathcal{C},\tau)$ is separately continuous we conclude that $b(ab)^pa$ is an isolated point in $(\mathcal{C},\tau)$ for every integer $p\geqslant 0$. Similarly, formula (\ref{eq-3.1}) implies that the equation $a\cdot X\cdot b=b^n(ab)^{p}a^n$ has the unique solution $X=b^{n-1}(ab)^pa^{n-1}$ for every nonnegative integer $p$ and every integer $n>1$. Then by induction we get that the separate continuity of the semigroup operation in $(\mathcal{C},\tau)$ implies that $b^{n+1}(ab)^pa^{n+1}$ is an isolated point in the topological space $(\mathcal{C},\tau)$ for all nonnegative integers $n$ and $p$.

We fix arbitrary distinct nonnegative integers $n$ and $m$. We can assume without loss of generality that $n<m$. In the case when $m<n$ the proof is similar. Since by Remark~\ref{remark-2.10a}$(i)$ we have that the map $\varphi_{m-n,0}\colon\mathcal{C}\rightarrow\mathcal{C}$ defined by the formula $\varphi_{m-n,0}(x)=b^{m-n}\cdot x$ is injective and by the previous part of the proof, the point $b^m(ab)^pa^m$ is isolated in $(\mathcal{C},\tau)$ for every nonnegative integer $p$, we conclude that the separate continuity of the semigroup operation in $(\mathcal{C},\tau)$ implies that $b^n(ab)^pa^m$ is an isolated point in the topological space $(\mathcal{C},\tau)$ for every nonnegative integer $p$.
\end{proof}

Since every \v{C}ech complete space (and hence every locally compact space) is Baire, Theorem~\ref{theorem-3.8} implies Corollaries~\ref{corollary-3.9} and \ref{corollary-3.10}.

\begin{corollary}\label{corollary-3.9}
Every Hausdorff \v{C}ech complete (locally compact) topology $\tau$ on $\mathcal{C}$ such that $(\mathcal{C},\tau)$ is a Hausdorff semitopological semigroup is discrete.
\end{corollary}

\begin{corollary}\label{corollary-3.10}
Every Hausdorff Baire topology (and hence \v{C}ech complete or locally compact topology) $\tau$ on $\mathcal{C}$ such that $(\mathcal{C},\tau)$ is a Hausdorff topological semigroup is discrete.
\end{corollary}

The following example implies that there exists a Tychonoff nondiscrete topology $\tau_p$ on the semigroup $\mathcal{C}$ such that $(\mathcal{C},\tau_p)$ is a topological semigroup.

\begin{example}\label{example-3.10}
Let $p$ be a fixed prime number. We define a topology $\tau_p$ on the semigroup $\mathcal{C}$ by the base $\mathscr{B}_p\!\!\left(b^i(ab)^ka^j\right)= \left\{U_{\alpha}\!\!\left(b^i(ab)^ka^j\right)\mid \alpha=1,2,3,\ldots\right\}$ at every point $b^i(ab)^ka^j\in\mathcal{C}$, where $U_{\alpha}\!\!\left(b^i(ab)^ka^j\right)= \left\{b^i(ab)^{k+\lambda \cdot p^{\alpha}}a^j\mid \lambda=1,2,3,\ldots\right\}$. Simple verifications show that the topology $\tau_p$ on $\mathcal{C}$ is generated by the following metric:
\begin{equation*}
    d\left(b^{i_1}(ab)^{k_1}a^{j_1},b^{i_2}(ab)^{k_2}a^{j_2}\right)=
\left\{
  \begin{array}{ll}
    0,     & \hbox{if~} i_1=i_2, k_1=k_2 \hbox{~and~} j_1=j_2;\\
    2^{s}, & \hbox{if~} i_1=i_2, k_1\neq k_2 \hbox{~and~} j_1=j_2;\\
    1, & \hbox{otherwise,}
  \end{array}
\right.
\end{equation*}
where $s$ is the largest of $p$ which divides $|k_1-k_2|$. This implies that $(\mathcal{C},\tau_p)$ is a Tychonoff space. Also, it is easy to see that $U_{\alpha}\!\!\left(b^i(ab)^ka^j\right)$ is a closed subset of the topological space $(\mathcal{C},\tau_p)$, for every $b^i(ab)^ka^j\in\mathcal{C}$ and every positive integer $\alpha$, and hence $(\mathcal{C},\tau_p)$ is a $0$-dimensional topological space (i.e., $(\mathcal{C},\tau_p)$ has a base which consists of open-and-closed subsets). We observe that the topological space $(\mathcal{C},\tau_p)$ doesn't contain any isolated points.

For every positive integer $\alpha$ and arbitrary elements $b^k(ab)^la^m$ and $b^n(ab)^ta^q$ of the semigroup $\mathcal{C}$, formula (\ref{eq-2.1}) implies that the following conditions hold:
\begin{itemize}
  \item[$(i)$] if $m<n$ then $U_{\alpha}\!\!\left(b^k(ab)^la^m\right)\cdot U_{\alpha}\!\!\left(b^n(ab)^ta^q\right)\subseteq U_{\alpha}\!\!\left(b^{k+n-m}(ab)^ta^q\right)$;
  \item[$(ii)$] if $m=n\neq 0$ then $U_{\alpha}\!\!\left(b^k(ab)^la^m\right)\cdot U_{\alpha}\!\!\left(b^n(ab)^ta^q\right)\subseteq U_{\alpha}\!\!\left(b^{k}(ab)^{l+t+1}a^q\right)$;
  \item[$(iii)$] if $m=n=0$ then $U_{\alpha}\!\!\left(b^k(ab)^la^m\right)\cdot U_{\alpha}\!\!\left(b^n(ab)^ta^q\right)\subseteq U_{\alpha}\!\!\left(b^{k}(ab)^{l+t}a^q\right)$; \; and
  \item[$(iv)$] if $m>n$ then $U_{\alpha}\!\!\left(b^k(ab)^la^m\right)\cdot U_{\alpha}\!\!\left(b^n(ab)^ta^q\right)\subseteq U_{\alpha}\!\!\left(b^k(ab)^la^{q+m-n}\right)$.
\end{itemize}
Therefore $(\mathcal{C},\tau_p)$ is a topological semigroup.
\end{example}


\section{On the closure and embedding of the semitopological semigroup $\mathcal{C}$}\label{s4}

In the case of the bicyclic semigroup $\mathcal{B}(a,b)$ we have that if a topological semigroup $S$ contains $\mathcal{B}(a,b)$ then the nonempty remainder of $\mathcal{B}(a,b)$  under the closure in $S$ is an ideal in $\operatorname{cl}_S(\mathcal{B}(a,b))$ (see \cite{EberhartSelden1969}). This immediately follows from that facts that the bicyclic semigroup $\mathcal{B}(a,b)$ admits only the discrete topology which turns $\mathcal{B}(a,b)$ into a Hausdorff semitopological semigroup and that the equations $A\cdot X=B$ and $X\cdot A=B$ have finitely many solutions in $\mathcal{B}(a,b)$ (see \cite[Proposition~1]{BertmanWest1976} and \cite[Lemma~I.1]{EberhartSelden1969}).

The following example shows that the semigroup $\mathcal{C}$ with the discrete topology does not have similar ``properties of the closure'' as the bicyclic semigroup.

\begin{example}\label{example-4.1}
It well known that each element of the bicyclic semigroup $\mathcal{B}(a,b)$ is uniquely expressible as $b^ia^j$, where $i$ and $j$ are nonnegative integers. Since all elements of the semigroup have similar expressibility we shall denote later the elements of the bicyclic semigroup by underlining $\underline{b^ia^j}$.

We define a map $\pi\colon\mathcal{C}\rightarrow\mathcal{B}(a,b)$ by the formula $\pi(b^i(ab)ka^j)=\underline{b^ia^j}$. Simple verifications and formula (\ref{eq-2.1}) show that thus defined map $\pi$ is a homomorphism. We extend the semigroup operation from the semigroups $\mathcal{C}$ and $\mathcal{B}(a,b)$ on $S=\mathcal{C}\sqcup\mathcal{B}(a,b)$ in the following way:
\begin{equation*}
    b^k(ab)^la^m\star \underline{b^na^q}=
\left\{
  \begin{array}{ll}
    \underline{b^{k+n-m}a^q}, & \hbox{if~} m<n;\\
    \underline{b^{k}a^q},     & \hbox{if~} m=n;\\
    b^k(ab)^la^{q+m-n},       & \hbox{if~} m>n
  \end{array}
\right.
\end{equation*}
and
\begin{equation*}
    \underline{b^ka^m}\star b^n(ab)^pa^q=
\left\{
  \begin{array}{ll}
    b^{k+n-m}(ab)^pa^q,       & \hbox{if~} m<n;\\
    \underline{b^{k}a^q},     & \hbox{if~} m=n;\\
    \underline{b^ka^{q+m-n}}, & \hbox{if~} m>n.
  \end{array}
\right.
\end{equation*}
A routine check of all 118 cases and their compatibility shows that such a binary operation is associative.

Now, we define the topology $\tau$ on the semigroup $S$ in the following way:
\begin{itemize}
  \item[$(i)$] all elements of the semigroup $\mathcal{C}$ are isolated points in  $(S,\tau)$; \; and
  \item[$(ii)$] the family $\mathscr{B}\left(\underline{b^ia^j}\right)=\left\{ U_n\!\left(\underline{b^ia^j}\right)\mid n=1,2,3,\ldots\right\}$, where
      \begin{equation*}
      U_n\!\left(\underline{b^ia^j}\right)=\left\{\underline{b^ia^j}\right\} \cup\left\{b^i(ab)^ka^j\in\mathcal{C}\mid k=n,n+1,n+2,\ldots\right\},
      \end{equation*}
      is a base of the topology $\tau$ at the point $\underline{b^ia^j}\in \mathcal{B}(a,b)$.
\end{itemize}

Simple verifications show that $(S,\tau)$ is a Hausdorff $0$-dimensional scattered locally compact metrizable space.
\end{example}

\begin{proposition}\label{proposition-4.2}
$(S,\tau)$ is a topological semigroup.
\end{proposition}

\begin{proof}
The definition of the topology $\tau$ on $S$ implies that it suffices to show that the semigroup operation in $(S,\tau)$ is continuous in the following three cases:
\begin{itemize}
  \item[1)] $\underline{b^ia^k}\star\underline{b^ma^p}$;
  \item[2)] $\underline{b^ia^k}\star b^m(ab)^na^p$; \; and
  \item[3)] ${b^i(ab)^la^k}\star\underline{b^ma^p}$.
\end{itemize}

In case 1) we get that
\begin{equation*}
    \underline{b^ia^k}\star\underline{b^ma^p}=
\left\{
  \begin{array}{ll}
    \underline{b^{i-k+m}a^p},   & \hbox{if~} k<m;\\
    \underline{b^ia^p},         & \hbox{if~} k=m;\\
    \underline{b^{i}a^{k-m+p}}, & \hbox{if~}, k>m,
  \end{array}
\right.
\end{equation*}
and for every positive integer $u$ the following statements hold:
\begin{itemize}
  \item[$a)$] if $k<m$ then $U_u\!\left(\underline{b^ia^k}\right)\star U_u\!\left(\underline{b^ma^p}\right)\subseteq U_u\!\left(\underline{b^{i-k+m}a^p}\right)$;
  \item[$b)$] if $k=m$ then $U_u\!\left(\underline{b^ia^k}\right)\star U_u\!\left(\underline{b^ma^p}\right)\subseteq U_u\!\left(\underline{b^ia^p}\right)$;
  \item[$c)$] if $k>m$ then $U_u\!\left(\underline{b^ia^k}\right)\star U_u\!\left(\underline{b^ma^p}\right)\subseteq U_u\!\left(\underline{b^{i}a^{k-m+p}}\right)$.
\end{itemize}

In case 2) we have that
\begin{equation*}
    \underline{b^ia^k}\star b^m(ab)^na^p=
\left\{
  \begin{array}{ll}
    {b^{i-k+m}(ab)^na^p},       & \hbox{if~} k<m;\\
    \underline{b^ia^p},         & \hbox{if~} k=m;\\
    \underline{b^{i}a^{k-m+p}}, & \hbox{if~}, k>m,
  \end{array}
\right.
\end{equation*}
and hence for every positive integer $u$ the following statements hold:
\begin{itemize}
  \item[$a)$] if $k<m$ then $U_u\!\left(\underline{b^ia^k}\right)\star \left\{b^m(ab)^na^p\right\}= \left\{b^{i-k+m}(ab)^na^p\right\}$;
  \item[$b)$] if $k=m$ then $U_u\!\left(\underline{b^ia^k}\right)\star \left\{b^m(ab)^na^p\right\}\subseteq U_u\!\left(\underline{b^ia^p}\right)$;
  \item[$c)$] if $k>m$ then $U_u\!\left(\underline{b^ia^k}\right)\star \left\{b^m(ab)^na^p\right\}\subseteq U_u\!\left(\underline{b^{i}a^{k-m+p}}\right)$.
\end{itemize}

In case 3) we have that
\begin{equation*}
    {b^i(ab)^la^k}\star\underline{b^ma^p}=
\left\{
  \begin{array}{ll}
    \underline{b^{i-k+m}a^p}, & \hbox{if~} k<m;\\
    \underline{b^ia^p},       & \hbox{if~} k=m;\\
    b^{i}(ab)^la^{k-m+p},     & \hbox{if~}, k>m.
  \end{array}
\right.
\end{equation*}
Then for every positive integer $u$ the following statements hold:
\begin{itemize}
  \item[$a)$] if $k<m$ then $\left\{b^i(ab)^la^k\right\}\star U_u\!\left(\underline{b^ma^p}\right) \subseteq \left\{b^{i-k+m}(ab)^na^p\right\}$;
  \item[$b)$] if $k=m$ then $\left\{b^i(ab)^la^k\right\}\star U_u\!\left(\underline{b^ma^p}\right) \subseteq U_u\!\left(\underline{b^ia^p}\right)$;
  \item[$c)$] if $k>m$ then $\left\{b^i(ab)^la^k\right\}\star U_u\!\left(\underline{b^ma^p}\right)=\left\{b^{i}(ab)^la^{k-m+p}\right\}$.
\end{itemize}
This completes the proof of the proposition.
\end{proof}

The following example shows that the semigroup $\mathcal{C}$ with the discrete topology may has similar closure in a topological semigroup as the bicyclic semigroup.

\begin{example}\label{example-4.3}
Let $S$ be the semigroup $\mathcal{C}$ with adjoined zero $0$. We determine the topology $\tau$ on the semigroup $S$ in the following way:
\begin{itemize}
  \item[$(i)$] All elements of the semigroup $\mathcal{C}$ are isolated points in  $(S,\tau)$; \; and
  \item[$(ii)$] The family $\mathscr{B}(0)=\left\{U_n(0)\mid n=1,2,3,\ldots\right\}$, where $U_n(0)=\{0\} \cup\left\{b^i(ab)^ka^j\in\mathcal{C}\mid i,j\geqslant n\right\}$, is a base of the topology $\tau$ at the zero $0$.
\end{itemize}

Simple verifications show that $(S,\tau)$ is a Hausdorff $0$-dimensional scattered space.

Since all elements of the semigroup $\mathcal{C}$ are isolated points in $(S,\tau)$ we conclude that it is sufficient to show that the semigroup operation in $(S,\tau)$ is continuous in the following cases:
\begin{equation*}
0\cdot 0, \qquad 0\cdot b^m(ab)^na^p, \qquad \hbox{and} \qquad b^m(ab)^na^p\cdot 0.
\end{equation*}
Since the following assertions hold for each positive integer $i$:
\begin{itemize}
  \item[$(i)$] $U_i(0)\cdot U_i(0)\subseteq U_i(0)$;
  \item[$(ii)$] $U_{i+m}(0)\cdot \{b^m(ab)^na^p\}\subseteq U_i(0)$;
  \item[$(iii)$] $\{b^m(ab)^na^p\}\cdot U_{i+p}(0)\subseteq U_i(0)$,
\end{itemize}
we conclude that $(S,\tau)$ is a topological semigroup.
\end{example}

\begin{remark}\label{remark-4.4}
We observe that we can show that for the discrete semigroup $\mathcal{C}$ cases of closure of $\mathcal{C}$ in topological semigroups proposed in \cite{EberhartSelden1969} for the bicyclic semigroup can be realized in the following way: we identify the element $b^ia^j$ of the bicyclic semigroup with the subset $\mathcal{C}_{i,j}$ of the semigroup $\mathcal{C}$.

We don't know the answer to the following question: \emph{Does there exist a topological semigroup $S$ which contains $\mathcal{C}$ as a dense subsemigroup such that $S\setminus\mathcal{C}\neq\varnothing$ and $\mathcal{C}$ is an ideal of $S$?}
\end{remark}

The following proposition describes the closure of the semigroup $\mathcal{C}$ in an arbitrary semitopological semigroup.

\begin{proposition}\label{proposition-4.5}
Let $S$ be a Hausdorff semitopological semigroup which contains $\mathcal{C}$ as a dense subsemigroup. Then there exists a countable family $\mathscr{U}=\left\{U_{\mathcal{C}_{i,j}}\mid i,j=0,1,2,3,\ldots\right\}$ of open disjunctive subsets of the topological space $S$ such that $\mathcal{C}_{i,j}\subseteq U_{\mathcal{C}_{i,j}}$ for all nonnegative integers $i$ and $j$.
\end{proposition}

\begin{proof}
When $S=\mathcal{C}$ the statement of the proposition follows from Theorem~\ref{theorem-3.5}. Hence we can assume that $S\neq\mathcal{C}$.

First, we observe that formulae (\ref{eq-3.2}) and (\ref{eq-3.3}) imply that for left and right translations $\lambda_{ab}\colon S\rightarrow S\colon x\mapsto ab\cdot x$ and $\rho_{ab}\colon S\rightarrow S\colon x\mapsto x\cdot ab$ of the semigroup $S$ their sets of fixed points $\operatorname{Fix}(\lambda_{ab})$ and $\operatorname{Fix}(\rho_{ab})$ are non-empty and moreover
\begin{align*}
    & \bigcup\left\{\mathcal{C}_{i,j}\mid i=0,1,2,3,\ldots,\; j=1,2,3,\ldots\right\}\subseteq\operatorname{Fix}(\rho_{ab}); \qquad \hbox{and}\\
    & \bigcup\left\{\mathcal{C}_{i,j}\mid i=1,2,3,\ldots, \; j=0,1,2,3,\ldots\right\}\subseteq\operatorname{Fix}(\lambda_{ab}).
\end{align*}
Also, formulae (\ref{eq-3.1a}) and (\ref{eq-3.1b}) imply that for every positive integer $n$ the left and right translations $\lambda_{b^na^n}\colon S\rightarrow S\colon x\mapsto b^na^n\cdot x$ and $\rho_{b^na^n}\colon S\rightarrow S\colon x\mapsto x\cdot b^na^n$ of the semigroup $S$ have non-empty sets of fixed points $\operatorname{Fix}(\lambda_{b^na^n})$ and $\operatorname{Fix}(\rho_{b^na^n})$, and moreover
\begin{align*}
    & \bigcup\left\{\mathcal{C}_{i,j}\mid i=0,1,2,3,\ldots, \; j=n+1,n+2,n+3,\ldots\right\}\subseteq\operatorname{Fix}(\rho_{b^na^n}); \qquad \hbox{and}\\
    & \bigcup\left\{\mathcal{C}_{i,j}\mid i=n+1,n+2,n+3,\ldots, \; j=0,1,2,3,\ldots\right\}\subseteq\operatorname{Fix}(\lambda_{b^na^n}).
\end{align*}

Then the Hausdorffness of $S$, separate continuity of the semigroup operation on $S$ and Exercise~1.5.C of \cite{Engelking1989} imply that $\operatorname{Fix}(\lambda_{ab})$, $\operatorname{Fix}(\rho_{ab})$, $\operatorname{Fix}(\lambda_{b^na^n})$ and $\operatorname{Fix}(\rho_{b^na^n})$ are closed non-empty subset of $S$, for every positive integer $n$, and hence are retracts of $S$.

Now, since $\mathcal{C}_{0,0}\subseteq S\setminus \left(\operatorname{Fix}(\lambda_{ab})\cup \operatorname{Fix}(\rho_{ab})\right)$ we conclude that there exists an open subset $U_{\mathcal{C}_{0,0}}= S\setminus \left(\operatorname{Fix}(\lambda_{ab})\cup \operatorname{Fix}(\rho_{ab})\right)$ which contains the set $\mathcal{C}_{0,0}$ and $\mathcal{C}_{i,j}\cap U_{\mathcal{C}_{0,0}}=\varnothing$ for all nonnegative integers $i,j$ such that $i+j>0$.

Since the semigroup operation in $S$ is separately continuous we conclude that the map $\lambda_{a}\colon S\rightarrow S\colon x\mapsto a\cdot x$ is continuous, and hence $U_{\mathcal{C}_{1,0}}= \lambda_{a}^{-1}\left( U_{\mathcal{C}_{0,0}}\right)\setminus\left(\operatorname{Fix}(\rho_{ab})\cup \operatorname{Fix}(\lambda_{ba})\right)$ is an open subset of $S$. It is obvious that $\mathcal{C}_{1,0}\subseteq U_{\mathcal{C}_{1,0}}$. We claim that $U_{\mathcal{C}_{1,0}}\cap U_{\mathcal{C}_{0,0}}=\varnothing$. Suppose to the contrary that there exists $x\in S$ such that $x\in U_{\mathcal{C}_{1,0}}\cap U_{\mathcal{C}_{0,0}}$. Since $\operatorname{Fix}(\lambda_{ba})$ and $\operatorname{Fix}(\rho_{ba})$ are closed subsets of $S$ we conclude that there exists $(ab)^i\in U_{\mathcal{C}_{1,0}}\cap U_{\mathcal{C}_{0,0}}$. Then we have that $\lambda_{a}((ab)^i)=a\cdot (ab)^i=a\notin U_{\mathcal{C}_{0,0}}$, a contradiction. The obtained contradiction implies that $U_{\mathcal{C}_{1,0}}\cap U_{\mathcal{C}_{0,0}}=\varnothing$.

Also, the continuity of the right shift $\rho_{b}\colon S\rightarrow S\colon x\mapsto x\cdot b$ implies that $U_{\mathcal{C}_{0,1}}=\rho_{b}^{-1}\left( U_{\mathcal{C}_{0,0}}\right)\setminus\left(\operatorname{Fix}(\lambda_{ab})\cup \operatorname{Fix}(\rho_{ba})\right)$ is an open neighbourhood of the set $\mathcal{C}_{0,1}$ in $S$. Similar arguments as in the previous case imply that $U_{\mathcal{C}_{0,1}}\cap U_{\mathcal{C}_{0,0}}=\varnothing$.

Suppose that there exists $x\in S$ such that $x\in U_{\mathcal{C}_{1,0}}\cap U_{\mathcal{C}_{0,1}}$. If $x\in\mathcal{C}$ then $x=b(ab)^p$ for some nonnegative integer $p$. Then we have that $\rho_{b}(x)=x\cdot b=b(ab)^p\cdot b=b^2\notin U_{\mathcal{C}_{0,0}}$. If $x\in U_{\mathcal{C}_{1,0}}\setminus\mathcal{C}$ then every open neighbourhood $V(x)$ of the point $x$ in the topological space $S$ contains infinitely many points of the form $b(ab)^p\in\mathcal{C}$. Then we have that $\rho_{b}(V(x))\ni b^2$. The obtained contradiction implies that $U_{\mathcal{C}_{1,0}}\cap U_{\mathcal{C}_{0,1}}=\varnothing$.

We put $U_{\mathcal{C}_{1,1}}= \left(\rho_{b}^{-1}\left(U_{\mathcal{C}_{1,0}}\right) \cap \lambda_{a}^{-1}\left(U_{\mathcal{C}_{0,1}}\right)\right)\setminus \left(\operatorname{Fix}(\lambda_{ba})\cup \operatorname{Fix}(\rho_{ba})\right)$.
Then $U_{\mathcal{C}_{1,1}}$ is an open subset of the topological space $S$ such that $\mathcal{C}_{1,1}\subseteq U_{\mathcal{C}_{1,1}}$. Similar arguments as in the previous cases imply that $U_{\mathcal{C}_{1,1}}\cap U_{\mathcal{C}_{0,1}}= U_{\mathcal{C}_{1,0}}\cap U_{\mathcal{C}_{1,1}}=U_{\mathcal{C}_{1,1}}\cap U_{\mathcal{C}_{0,0}}=\varnothing$.

Next, we use induction for constructing the family $\mathscr{U}$. Suppose that for some positive integer $n\geqslant 1$ we have already constructed the family $\mathscr{U}_n=\left\{U^{\prime}_{\mathcal{C}_{i,j}}\mid i,j=0,1,\ldots, n\right\}$ of open disjunctive subsets of the topological space $S$ with the property $\mathcal{C}_{i,j} \subseteq U_{\mathcal{C}_{i,j}}$, for all $i,j=0,1,\ldots, n$. We shall construct the family $\mathscr{U}_{n+1}=\left\{U_{\mathcal{C}_{i,j}}\mid i,j=0,1,\ldots, n,n+1\right\}$ in the following way. For all $i,j\leqslant n$ we put $U_{\mathcal{C}_{i,j}}=U^{\prime}_{\mathcal{C}_{i,j}}\in\mathscr{U}_n$ and
\begin{equation*}
\begin{split}
& U_{\mathcal{C}_{0,n+1}}=\rho_{b}^{-1}\left( U_{\mathcal{C}_{0,n}}\right)\setminus \left(\operatorname{Fix}(\lambda_{ab})\cup \operatorname{Fix}(\rho_{b^{n+1}a^{n+1}})\right); \\
& U_{\mathcal{C}_{1,n+1}}=\rho_{b}^{-1}\left( U_{\mathcal{C}_{1,n}}\right)\setminus \left(\operatorname{Fix}(\lambda_{ba})\cup \operatorname{Fix}(\rho_{b^{n+1}a^{n+1}})\right);\\
& \cdots\qquad\cdots\qquad\cdots\qquad\cdots\qquad\cdots\qquad\cdots\qquad\cdots\\
& U_{\mathcal{C}_{n,n+1}}=\rho_{b}^{-1}\left( U_{\mathcal{C}_{n-1,n}}\right)\setminus \left(\operatorname{Fix}(\lambda_{b^na^n})\cup \operatorname{Fix}(\rho_{b^{n+1}a^{n+1}})\right);\\
& U_{\mathcal{C}_{n+1,0}}=\lambda_{a}^{-1}\left( U_{\mathcal{C}_{n,0}}\right) \setminus \left(\operatorname{Fix}(\rho_{ab})\cup \operatorname{Fix}(\lambda_{b^{n+1}a^{n+1}})\right); \\
& U_{\mathcal{C}_{n+1,1}}=\lambda_{a}^{-1}\left( U_{\mathcal{C}_{n,1}}\right) \setminus \left(\operatorname{Fix}(\rho_{ba})\cup \operatorname{Fix}(\lambda_{b^{n+1}a^{n+1}})\right); \\
& \cdots\qquad\cdots\qquad\cdots\qquad\cdots\qquad\cdots\qquad\cdots\qquad\cdots\\
& U_{\mathcal{C}_{n+1,n}}=\lambda_{a}^{-1}\left( U_{\mathcal{C}_{n,n}}\right) \setminus \left(\operatorname{Fix}(\rho_{b^na^n})\cup \operatorname{Fix}(\lambda_{b^{n+1}a^{n+1}})\right); \\
& U_{\mathcal{C}_{n+1,n+1}}=\left( \rho_{b}^{-1}\left(U_{\mathcal{C}_{n+1,n}}\right)\cap \lambda_{a}^{-1}\left(U_{\mathcal{C}_{n,n+1}}\right)\right)\setminus \left(\operatorname{Fix}(\rho_{b^{n+1}a^{n+1}})\cup \operatorname{Fix}(\lambda_{b^{n+1}a^{n+1}})\right).
\end{split}
\end{equation*}

Similar arguments as in previous case imply that $\mathscr{U}_{n+1}$ is a family of open disjunctive subsets of the topological space $S$ with the property $\mathcal{C}_{i,j} \subseteq U_{\mathcal{C}_{i,j}}$, for all $i,j=0,1,\ldots, n+1$.

Next, we put $\displaystyle\mathscr{U}=\bigcup_{n=0}^\infty \mathscr{U}_n$. It is easy to see that the family $\mathscr{U}$ is as required. This completes the proof of the proposition.
\end{proof}


It well known that if a topological semigroup $S$ is a continuous image of a topological semigroup $T$ such that $T$ is embeddable into a compact topological semigroup, then the semigroup $S$ is not necessarily embeddable into a compact topological semigroup. For example, the bicyclic semigroup $\mathcal{B}(a,b)$ does not embed into any compact topological semigroup, but $\mathcal{B}(a,b)$ admits only discrete semigroup topology and $\mathcal{B}(a,b)$ is a continuous image of the free semigroup $F_2$ of the rank 2 (i.e., generated by two elements) with the discrete topology. Moreover, the semigroup $F_2$ with adjoined zero $0$ admits a compact Hausdorff semigroup topology $\tau_c$: all elements of $F_2$ are isolated points and the family $\mathscr{B}_0=\{U_n\mid n=1,2,3,\ldots\}$, where the set $U_n$ consists of zero $0$ and all words of length $\geqslant n$. Therefore it is natural to ask the following: \emph{Does there exist a Hausdorff compact topological semigroup $S$ which contains the semigroup $\mathcal{C}$?} The following theorem gives a negative answer to this question.

\begin{theorem}\label{theorem-5.1}
There does not exist a Hausdorff topological semigroup $S$ with a countably compact square $S\times S$ such that $S$ contains $\mathcal{C}$ as a subsemigroup.
\end{theorem}

\begin{proof}
Suppose to the contrary that there exists a Hausdorff topological semigroup $S$ with a countably compact square $S\times S$ which contains $\mathcal{C}$ as a subsemigroup. Then since the closure of a subsemigroup $\mathcal{C}$ in a topological semigroup $S$ is a subsemigroup of $S$ (see \cite[Vol.~1, p.~9]{CHK}) we conclude that Theorem~3.10.4 from \cite{Engelking1989} implies that without loss of generality we can assume that $\mathcal{C}$ is a dense subsemigroup of the topological semigroup $S$. We consider the sequence $\left\{(a^n,b^n)\right\}_{n=1}^\infty$ in $\mathcal{C}\times\mathcal{C}\subseteq S\times S$. Since $S\times S$ is countably compact we conclude that this sequence has an accumulation point $(x;y)\in S\times S$. Since $a^nb^n=ab$, the continuity of the semigroup operation in $S$ implies that $xy=ab$. By Proposition~\ref{proposition-4.5} there exists an open neighbourhood $U(ab)$ of the point $ab$ in $S$ such that $U(ab)\cap\mathcal{C}\subseteq\mathcal{C}_{0,0}$. Then the continuity of the semigroup operation in $S$ implies that there exist open neighbourhoods $U(x)$ and $U(y)$ of the points $x$ and $y$ in $S$ such that $U(x)\cdot U(y)\subseteq U(ab)$. Next, by the countable compactness of $S\times S$ we conclude that $S$ is countably compact, too, as a continuous image of $S\times S$ under the projection, and this implies that $x$ and $y$ are accumulation points of the sequences $\{a^n\}_{n=1}^\infty$ and $\{b^n\}_{n=1}^\infty$ in $S$, respectively. Then there exist positive integers $i$ and $j$ such that $a^i\in U(x)$, $b^j\in U(y)$ and $j>i$. Therefore we get that
\begin{equation*}
    a^i\cdot b^j=b^{j-i}\in\left(U(x)\cdot U(y)\right)\cap\mathcal{C}\subseteq \left(U(ab)\right)\cap\mathcal{C}\subseteq \mathcal{C}_{0,0},
\end{equation*}
which is a contradiction. The obtained contradiction implies the statement of the theorem.
\end{proof}

Theorem~\ref{theorem-5.1} implies the following corollaries:

\begin{corollary}\label{corollary-5.2}
There does not exist a Hausdorff compact topological semigroup which contains $\mathcal{C}$ as a subsemigroup.
\end{corollary}

\begin{corollary}\label{corollary-5.3}
There does not exist a Hausdorff sequentially compact topological semigroup which contains $\mathcal{C}$ as a subsemigroup.
\end{corollary}

We recall that the Stone-\v{C}ech compactification of a Tychonoff space $X$ is a compact Hausdorff space $\beta X$ containing $X$ as a dense subspace so that each continuous map $f\colon X\rightarrow Y$ to a compact Hausdorff space $Y$ extends to a continuous map $\overline{f}\colon \beta X\rightarrow Y$ \cite{Engelking1989}.

\begin{theorem}\label{theorem-5.4}
There does not exist a Tychonoff topological semigroup $S$ with the pseudocompact square $S\times S$ which contains $\mathcal{C}$ as subsemigroup.
\end{theorem}

\begin{proof}
By Theorem~1.3 from \cite{BanakhDimitrova2010}, for any topological
semigroup $S$ with the pseudocompact square $S\times S$ the
semigroup operation $\mu\colon S\times S\rightarrow S$ extends to a
continuous semigroup operation $\beta\mu\colon \beta S\times\beta
S\rightarrow\beta S$, so S is a subsemigroup of the compact
topological semigroup $\beta S$. Therefore if $S$ contains the
semigroup $\mathcal{C}$ then $\beta S$ also contains the
semigroup $\mathcal{C}$ which contradicts Corollary~\ref{corollary-5.2}.
\end{proof}

\begin{theorem}\label{theorem-5.5}
The discrete semigroup $\mathcal{C}$ does not embed into a Hausdorff pseudocompact semitopological semigroup $S$ such that $\mathcal{C}$ is a dense subsemigroup of $S$ and $S\setminus\mathcal{C}$ is a left (right, two-sided) ideal of $S$.
\end{theorem}

\begin{proof}
Suppose to the contrary that there exists a Hausdorff pseudocompact semitopological semigroup $S$ which contains $\mathcal{C}$ as a dense discrete subsemigroup and $I=S\setminus\mathcal{C}$ is a left ideal of $S$. Then the set of solutions $\mathscr{S}$ of the equations $x\cdot ba=ba$ in $S$ is a subset of $\mathcal{C}$ and hence by the formula
\begin{equation*}
    b^k(ab)^la^m\cdot ba=
\left\{
  \begin{array}{ll}
    b^{k+1}a,          & \hbox{if~} m=0;\\
    b^{k}(ab)^{l+1}a,  & \hbox{if~} m=1;\\
    b^k(ab)^la^{m},    & \hbox{if~} m>1,\\
  \end{array}
\right.
\end{equation*}
we get that $\mathscr{S}=\mathcal{C}_{0,0}$. Since $ba$ is an isolated point in $S$ and $I$ is a left ideal of $S$ we conclude that the separate continuity of the semigroup operation of $S$ implies that the space $S$ contains a discrete open-and-closed subspace $\mathcal{C}_{0,0}$. This contradicts the pseudocompactness of $S$. The obtained contradiction implies the statement of the theorem. In the case of a right or a two-sided ideal the proof is similar.
\end{proof}

\begin{theorem}\label{theorem-5.6}
The semigroup $\mathcal{C}$ does not embed into a Hausdorff countably compact semitopological semigroup $S$ such that $\mathcal{C}$ is a dense subsemigroup of $S$ and $S\setminus\mathcal{C}$ is a left (right, two-sided) ideal of $S$.
\end{theorem}

\begin{proof}
Suppose to the contrary that there exists a Hausdorff countably compact semitopological semigroup $S$ which contains $\mathcal{C}$ as a dense subsemigroup and $I=S\setminus\mathcal{C}$ is a left ideal of $S$. Then the arguments presented in the proof of Theorem~\ref{theorem-5.5} imply that $\mathcal{C}_{0,0}$ is a closed subset of $S$, and hence by Theorem~3.10.4 of \cite{Engelking1989} is countably compact. Since $\mathcal{C}_{0,0}$ is countable we have that the space $\mathcal{C}_{0,0}$ is compact. Since every compact space is Baire, Lemma~\ref{lemma-3.7} implies that $\mathcal{C}_{0,0}$ is a discrete subspace of $S$. Then similar arguments as in the proof of Theorem~\ref{theorem-3.8} imply that $\mathcal{C}$, with the topology induced from $S$, is a discrete semigroup, which contradicts Theorem~\ref{theorem-5.5}. The obtained contradiction implies the statement of the theorem.
\end{proof}

\section*{Acknowledgements}

This research was supported by the Slovenian Research Agency grants
P1-0292-0101,  J1-4144-0101 and BI-UA/11-12/001.


\end{document}